\documentclass[11pt,a4paper]{article}

\usepackage[utf8]{inputenc} 
\usepackage{graphicx}
\usepackage{amsmath}
\usepackage{amssymb}
\usepackage{geometry}
\usepackage{caption}
\usepackage{subcaption}
\usepackage{abraces}
\usepackage{amsthm}
\usepackage{bbm}
\usepackage{tikz-cd}
\usepackage{hyperref}

\newgeometry{vmargin={20mm}, hmargin={20mm,20mm}} % Minimum normal values for margins; otherwise lines may run off the page. 

\title{Periodic orbits in general Glass networks}
\author{Huy K. Vo \\ \\
\textit{Mathematical Institute, University of Oxford, Oxford OX2 6GG\/} \\
}

\newtheorem{theorem}{Theorem}[section]
\newtheorem{corollary}[theorem]{Corollary}
\newtheorem{lemma}[theorem]{Lemma}
\theoremstyle{definition}
\newtheorem{defn}[theorem]{Definition}

\newtheorem{prop}[theorem]{Proposition}
\newtheorem{remark}[theorem]{Remark}

\DeclareMathOperator{\relint}{relint}
\usepackage{esvect}

\begin{document}

\maketitle

\begin{abstract}
One of the most well-studied class of models of the original Glass network are the cyclic attractor in the orthants (a sequence of orthants where the flow from one orthant to another is unanimous), which was first defined and analysed by Glass and Pasternack in 1978. In that paper, the authors gave a complete classification of the topological features of the flow in a full-rank cyclic attractor, which is a cyclic attractor that cannot be contained in any sub-cube in the graph of orthants. 

In this paper, we will extend the definition of cyclic attractor to one generalisation of the Glass network, one that allows for multiple switching points in each variables, and give a complete classification of the topological features of the flow for any cyclic attractor, both in the extended network and the original network, including non full-rank ones. We will show that in any cyclic attractor, there is either a unique and asymptotically stable periodic orbit, or that all periodic orbits are degenerated.

\end{abstract}

\section{Introduction}

A Glass network is a piecewise linear ODE system that models an interactive system where there are 'switching points': the underlying dynamic changes qualitatively when a certain variable pass over a threshold. For a Glass network with uniform decay rate, the trajectories are now piecewise straight line segment, and we can derived an explicit discrete map between walls, in which trajectories starting at a point on one wall will enter another wall. This particular feature also makes it being perceived as less challenging, (see \cite{Etienne_2009}, \cite{Snoussi}). Nevertheless, there are still many fine details that have not been fully examined before.

Among the different versions of Glass network, the simplest and best understood one is the binary Glass network, the original version, where assuming all thresholds have value $0$ (which can be obtained by an affine transformation), Glass and Pasternack in 1978, \cite{Glass_1978}, and subsequently Edwards in 2000, \cite{Edwards}, show that the map between walls will have a linear fractional form
\[
Mx = \frac{Ax}{1+\langle\psi,x\rangle}.
\]
Hence, trajectories that start on the same ray throught the origin remain on the same ray through the origin (proposition 5, \cite{Edwards}), and the task of finding periodic orbits is reduced to the task of finding fixed points on the Poincare section (which is a subset of a wall), which, in turn, is reduced down to finding eigenvectors of the linear map $A$. A full description of this process can be found in section 4, \cite{Edwards}.

The above analysis for a binary Glass network motivates us to examine a Glass network with uniform decay rate, but we now allow multiple thresholds in each variable (called a \textbf{general Glass network} in this paper). We choose to study this particular generalisation because it is a natural extension of the idea of a switching network, while still retain a nice geometric property of the trajectories, that they are piecewise straight segments.

With that aim in mind, we have found a beautiful technique to ``embed'' the equivalent of a general Glass network in a higher dimension binary Glass network. Applying this technique, together with one particular extension of the Perron-Frobenius theorem accredited to Birkhoff \cite{CMT}, we sucessfully extend the analysis of Glass and Pasternack for full-rank cyclic attractor to any cyclic attractor, in both the binary and general Glass network. This method could potentially be useful in extending many theoretical and generic results (such as those in \cite{Edwards} and \cite{EDWARDS2012666}) for the binary Glass network onto the general Glass network.

The hypothesis that in a cyclic attractor (of any decay rate, without alignment of focal points), there is at most one periodic orbit is still an open question, but hopefully this work will make a small contribution towards a complete understanding of that hypothesis.

Throughout this paper, the examples are all geometrical, that is, they are created by manually selecting thresholds and focal points. Such approach is not new, in face, it has been the basis of the works in \cite{EDWARDS2012666} and \cite{Ed2010} . This is a natural consequence of the geometrical approach and mindset that I have had when studying these models. A frequently used tool to analyse the Glass network is a State Transition Diagram. We will not go into this concept in detail in this paper, but it is still worth noting, since it explains the origin of terminologies like ``cycle in the regions'' and ``cyclic attractor''.

The paper is organised as follow: section 2 gives an introduction to the theory of Glass networks, the general version that we are studying, and the notations. Section 3 gives an overview of the main techniques and results we have obtained. Sections 4 and 5 provide proofs for our results, while section 6 discuss a few miscellaneous topics that are related to the practical application aspect of our technique.

\section{Preliminary setting}
Given a system that consists of $n$ variables $x_1$, $x_2$, $...$, $x_n$, each variable $x_i$ has $J_i$ thresholds $\theta_{i}^1 < \theta_{i}^2 <... <\theta_{i}^{J_i}$.
Throughout this paper, $\Tilde{\cdot}$ is a step function from $\mathbb{R} \setminus {\theta_{i}^j}$ to $\mathbb{Z}$ such that
\[
\Tilde{x_i} = \begin{cases}
    0 & \text{ if $x_i < \theta_{i,1}$} \\
    j & \text{ if $\theta_i^j<x_i < \theta_{i,j+1}$} \\
    j_i+1 & \text{ if $\theta_{i,j_i}<x_i$}
\end{cases}
\]
A region in the network would be of the form 
\[ 
\mathcal{O} = \prod_{i=1}^n (\theta_i^{j_i},\theta_i^{j_i +1}).
\]
Easily see that a region is an equivalent class of points with respect to $\Tilde{\cdot}$. Thus, $\Tilde{ \mathcal{O}}$ is defined by $\Tilde{x}$, for any point $x$ in $\mathcal{O}$.
A wall between two regions can be described as follow:
\[
\mathcal{W} = \prod_{i=1}^{k-1}(\theta_i^{j_i}, \theta_i^{j_i +1}) \times \{\theta_k^{j_k} \} \times \prod_{i=k+1}^{n}(\theta_i^{j_i}, \theta_i^{j_i +1}),
\]
where $1 \leq k \leq n$, and a vertice is of the form
\[ 
v = \prod_{i=1}^n \{ \theta_i^{j_i} \}.
\]

\begin{defn} A \textbf{Glass network} is a dynamical system of the form 
\[
\dot{y}_i = -y_i + F_i(\Tilde{y_1},\Tilde{y_2},\dots,\Tilde{y_n}),
\]
for $i\in \{1,2,\dots,n\}$. Thus, $F$ depends only on the region $y$ is currently in.
\end{defn}

If all the variables have only one threshold, then we call it the \textbf{binary Glass network}. A \textbf{general Glass network} would be still of the same form, but we allow the variables to have multiple thresholds. In fact, in this paper, whenever we use the term \textbf{general Glass network}, we mean a system where at least one variable has at least two thresholds. The term \textbf{Glass network} is used in this paper to indicate the combination of both definitions.

Then as long as the values $\Tilde{x_i}$ do not change, which is guaranteed if ${x}$ stays in the same region, the vector ${F} = (F_1,F_2,\dots,F_n)^T$ remains constant, and the trajectory will have the form
\[
{y}(t) = {F} + ({y}(0)-{F}) e^{-t}.
\]
From the above formula, the trajectory ${y}(t)$ approaches $F$ in a straight line. Thus, ${F}$ is called a focal point in $\mathbb{R}^n$ associated with a particular region in $\mathbb{R}^n$.

The dynamics on the walls, and on boundaries of the walls, are not immediately obvious from the above definition. It requires the usage of Filippov dynamics, but in our paper, we will not go into detail on what happens on the boundaries of the walls. Our setting generally involve a sequence (or cycle) of regions $\mathcal{O}_1 \rightarrow \mathcal{O}_2 \rightarrow \dots \rightarrow \mathcal{O}_k \rightarrow \mathcal{O}_{k+1} (= \mathcal{O}_1, \text{ in case of cycle})$, and $\mathcal{C} = \bigcup \mathcal{O}_i$. Let $(x_m, \theta_m^j)$ be the variable that switch and the threshold of the wall $\mathcal{W}_i$ between $\mathcal{O}_i$ and $\mathcal{O}_{i+1}$. Let ${f^i}$ and ${f^{i+1}}$ be the focal points of $\mathcal{O}_i$ and $\mathcal{O}_{i+1}$. If ${\Tilde{\mathcal{O}}_{{i+1}_m}} > {\Tilde{\mathcal{O}}_{i_m}}$, we require that both $f^i_m$ and $f^{i+1}_m$ are greater than $\theta_m^j$; conversely, if ${\Tilde{\mathcal{O}}_{{i+1}_m}} < {\Tilde{\mathcal{O}}_{i_m}}$, then both $f^i_m$ and $f^{i+1}_m$ are less than $\theta_m^j$. Then trajectories starting in $\mathcal{O}_i$ will move toward the wall $\mathcal{W}_i$, and trajectories starting in $\mathcal{O}_{i+1}$ will move away from the wall $\mathcal{W}_i$. Hence, we define the dynamics on the wall $\mathcal{W}_i$ to be a continuous extension of the dynamics in the region $\mathcal{O}_{i+1}$, thus bypassing the need for Filippov dynamics. 

In this paper, we assume that our system satisfies one additional condition, which is analogous to Condition 1 in \cite{Edwards}, that is, for any focal point ${F}$ and for any coordinate $i$, $F_i \neq \theta_{i}^j$, for all $j$.
This is necessary so that no trajectory will approach a wall asymptotically.

For a region $\mathcal{O}$ and its focal point $f$, we define the sets $I^0$, $I^+$, and $I^-$ as follow:
\begin{enumerate}
    \item $I^0 = \{i: \Tilde{f_i} = {\Tilde{\mathcal{O}}_i} \}$
    \item $I^+ = \{i: \Tilde{f_i} > {\Tilde{\mathcal{O}}_i} \}$
    \item $I^- = \{i: \Tilde{f_i} < {\Tilde{\mathcal{O}}_i} \}$.
\end{enumerate}

\subsection{Discrete maps, Periodic orbits, Poincare maps, and cyclic attractors}

Given a sequence of regions satisfies the above conditions, $\mathcal{O}_1 \rightarrow \mathcal{O}_2 \rightarrow \dots \rightarrow \mathcal{O}_k$, we do not always requires all trajectories starting in $\mathcal{O}_i$ to enter $\mathcal{O}_{i+1}$, but assume that there are such trajectories, we often consider the discrete map from the wall $\mathcal{W}_{i-1}$ to the wall $\mathcal{W}_{i}$. Denote this map $M_i$. Then, provided $x\in \mathcal{W}_{i-1}$, $x$ is in the domain of $M_i$, the trajectory starting at $x$ will cross the wall $\mathcal{W}_{i}$ at the point $M_i x$.

Now, given a cycle in the regions $\mathcal{C}$: $\mathcal{O}_1 \rightarrow \mathcal{O}_2 \rightarrow \dots \rightarrow \mathcal{O}_k \rightarrow \mathcal{O}_{k+1}= \mathcal{O}_{1}$, to find a periodic orbit of the cycle, we consider any wall along the cycle, without loss of generality, take $\mathcal{M}_0 = \mathcal{M}_k$ to be the Poincare section. Then the Poincare return map $M$ is the composition of the discrete maps $M_i$ along the cycle. A fixed point of $M$ on $\mathcal{M}_0$ corresponds to a periodic orbit in $\mathcal{C}$, and a stable fixed point of $M$ corresponds to a stable periodic orbit in $\mathcal{C}$.

We now come to the concept of ``cyclic attractor''. A cyclic attractor, in the binary network, is a cycle of regions such that trajectories starting in one region are all unambiguously enter the next region. A full-rank cyclic attractor is a cyclic atractor such that any trajectory travelling along it will have each of its variable switches at least once. This concept was first used by Glass in his 1978 paper \cite{Glass_1978}, in which he presented ``a complete classification of the topological features of the flows'' in a full-rank cyclic attractor, by the below theorem 

\begin{theorem}
\label{theorem_2.2}
Given an $n$ dimensional binary Glass network, in which there is a full-rank cyclic attractor $\mathcal{C}$, then one of the following two situations holds:
\begin{enumerate}
    \item There is a unique and stable limit cycle (periodic orbit) in phase space which passes through the orthants in the same sequence and order as the cyclic attractor in the state transition diagram. The trajectories in $\mathcal{C}$ asymptotically approach the limit cycle as $t \rightarrow \infty$.
    \item The trajectories through the points of orthants and boundaries represented by the cyclic attractor eventually approach the origin (the unique vertice in the binary case).
\end{enumerate}
\end{theorem}

This is a paraphrase of the main theorem in \cite{Glass_1978}; the conclusion in the second case has been slightly modified, since Edwards in his paper \cite{Edwards} has pointed out that this is finite time convergence, and the trajectory over a certain time limit is not very well-defined.

Hence, a natural extension of this concept to a general Glass network is as follow

\begin{defn}
    A cyclic attractor is a cycle of regions such that all trajectories starting in one region unambiguously enter the next region.
\end{defn}

\begin{prop}
\label{prop_2.4}
  Let $\mathcal{C}$ be a cyclic attractor in a Glass network, and let $\mathcal{O}$ be one of the region of $\mathcal{C}$. Then $|I^+ (\mathcal{O})\cup I^- (\mathcal{O})| = 1$. In other words, compare the coordinate of the focal point ${f}$ of $\mathcal{O}$ to points in $\mathcal{O}$, there is a unique variable of ${f}$ that switches.
\end{prop}

\begin{proof}
    Without loss of generality, let $\mathcal{O} = \mathcal{O}_1$, and $(x_m,\theta^j_m)$ be the variable that switches and the threshold of the wall $\mathcal{W}_1$ between $\mathcal{O}_1$ and $\mathcal{O}_2$. Clearly $m \in I^+ (\mathcal{O}_1)\cup I^- (\mathcal{O}_1)$.

    Assume that there is another variable $m' \in I^+ (\mathcal{O}_1)\cup I^- (\mathcal{O}_1)$. Without loss of generality, let $m' \in I^+ (\mathcal{O}_1)$. Choosing a point sufficiently close to the wall that bounded $\mathcal{O}_1$ from above in the $m'$ coordinate, the trajectory starting at that point will switch first in the $m'$ coordinate to a different wall from $\mathcal{W}_1$, and entering a different region than $\mathcal{O}_2$. But this is a contradiction. Hence, there can only be a unique variable in $I^+ (\mathcal{O})\cup I^- (\mathcal{O})$.
\end{proof}

Formally, we do not consider the dynamics of points on the boundaries of walls. But oftenly, we define the dynamics of points on these boundaries to be a continuous extension of the dynamics of points in the relative interior of the walls. However, there are cases which requires special care.

Suppose we have a discrete map $M_i$ from the wall $\mathcal{W}_{i-1}$ to $\mathcal{W}_{i}$. We could easily continuously extend this map as from $\overline{\mathcal{W}_{i-1}}$ to $\overline{\mathcal{W}_{i}}$. Then points on $\overline{\mathcal{W}_{i-1}} \cap \overline{\mathcal{W}_{i}}$ are \textbf{trivial fixed points} of the discrete map $M_i$. We often implicitly extend the map $M_i$ for analysing it. It has to be noted that if the dynamics on the overlapping walls are irreconcilable, then we consider the overlapping set to be distinct from each other. The intuition is that we are not really trying to extend the ODE systems to points on the boundaries of the wall; rather, we consider $\overline{\mathcal{W}_{i}}$ as a closure of ${\mathcal{W}_{i}}$ that is convenient for analytical purpose. To be more rigorous, one will need to consider the closure of the image over a homeomorphism into another vector space, but that would be an over complication.

Let $\mathcal{S}(\mathcal{C}) = \bigcup_{i=0}^{k-1} \overline{\mathcal{W}}_i$ be \textbf{the spine} of the cycle of regions $\mathcal{C}$ ($\mathcal{S}(\mathcal{C})$ may be empty) (we choose this terminology because it resembles the spine of a book). Then all points on $\mathcal{S}(\mathcal{C})$ are trivial fixed points of any Poincare return map $M$ (after continuous extension to boundaries of the walls of the cycle). This works perfectly well in the context of discrete maps, but there are potential problems in interpreting these points in the context of the Glass network, since they take no time at all ``travelling'' along the cycle $\mathcal{C}$.

Formally, we do not have to worry about the spine of a cycle, since they are not part of a cycle. However, what would happens if trajectories in the cycle converges to points on the spine? Edwards in his paper \cite{Edwards}, albeit in the context of the binary case, identified this as a case of ``finite time convergence'', and called these periodic orbits (periodic orbit that consists of a single point on the spine) \textbf{degenerate}, since they have period $0$ and are not well-defined as $t \rightarrow \infty$. In this paper, we shall treat the similar circumstances likewise; the analysis here should be easily adapted to past or future works that have extra assumptions on the dynamics of points on the boundaries of walls.

In the binary Glass network, any cyclic attractor will have the unique vertice $v$ (the intersection of all the threshold values) inside the spine. However, this is not necessarily the case in a cyclic attractor in a general Glass network. We call a cyclic attractor $\mathcal{C}$ to be \textbf{ideal} if $\mathcal{S} (\mathcal{C}) = \emptyset$ (since, among other reason, we do not have to worry about any degenerate periodic orbit), and \textbf{non-ideal} otherwise. We shall see, in latter section, that ideal cyclic attractors always have a unique and stable periodic orbit, regardless of its rank.

\begin{prop}
\label{prop_2.5}
    If $\mathcal{C}$ is a non-ideal cyclic attractor, there is at least one vertice $a$ of the network, such that $a$ belongs to all the closures of regions in $\mathcal{C}$.
\end{prop}

\begin{proof}
    One can choose any vertice on the closure of $\mathcal{S}(\mathcal{C})$.
\end{proof}

Thus, in the case of a non-ideal cyclic attractor, our approach is to take such a vertice as our centre of a binary Glass network of the same dimension that behaves identically on $\mathcal{C}$ to our original system. A detailed explanation of this approach is provided in subsection \ref{ssection_5.3}.

\subsection{Birkhoff generalisation of the Perron-Frobenius theorem}

A key result which we will frequently employ in this paper is an extension of the Perron-Frobenius theorem, that is attributed to Birkhoff. The statement is as follow:

\begin{theorem}
  \label{Birkhoff}
Let $K$ be a closed pointed cone, $d$ the Hilbert metric with respect to the cone $K$, and $A$ is a linear mapping, such that for every nonzero point $x$ of $K$, $A x$ lies in the relative interior of $K$ ($A$ is a ``positive mapping''). Then $A$ is a contraction mapping in $K$ with respect to the Hilbert metric $d$. Further, we have that:
\begin{enumerate}
    \item Given two sequences of non zero points in $K$: $\{x_n\}$, $\{y_n\}$, if $d(x_n,y_n) \rightarrow 0$ as $n\rightarrow \infty$, then $|| x_n/(||x_n||) - y_n/(||y_n||) || \rightarrow 0$ as $n\rightarrow \infty$.
    \item The metric space $(K\cap B_1(0),d)$ is complete.
\end{enumerate}
Consequently, convergences in Hilbert's metric implies convergences in direction. By the contraction mapping theorem, there exists a unique eigenvector $v$ of $A$ in $K$, such that every array in $K$ will converge to the array $(v)$ repeatedly apply $A$.  
\end{theorem}

For a full proof and discussion of the above theorem, see theorem 4.1, theorem 6.1 in \cite{CMT}, and theorem 2.3 in \cite{KOHLBERG1983104}.

Hilbert metric is in essence a metric between arrays (halflines) through the origin. The definition is as follow:
\begin{defn}
    Let $x$, $y$ be two points in a cone $K \subset \mathbb{R}^n$. The Hilbert metric between $x$ and $y$ is defined as
    \[
    d(x,y) = \log \frac{\max_i (x_i/y_i)}{\min_i (x_i/y_i)},
    \]where 
    \[
    \max_i (x_i/y_i) = \inf \{\lambda \geq 0: \lambda y -x \in K \},
    \]and
    \[
    \min_i (x_i/y_i) = \sup \{\lambda \geq 0: x- \lambda y  \in K \}.
    \]
\end{defn}

A full proof for Theorem \ref{Birkhoff} can be found in \cite{CMT}. A very brief summary of the proof is as follow: from the fact that $A$ is a positive mapping, it can be deduce that for all points $x$, $y$ in $K$, $d(Ax,Ay) < \infty$ ($d$ is the Hilbert metric, using cross ratio). Then, by decomposing the components into bases, or otherwise, we can obtain a finite upper bound for $d(Ax,Ay)$, for all $x,y \in K$. Call this upper bound $\Delta$. We can then obtain a contraction rate for $A$ in Hilbert metric that is strictly less than $1$ as follow:
\[
k_d(A) = \frac{\sqrt{e^\Delta}- 1}{\sqrt{e^\Delta}+1}.
\]
Then the conclusion of Theorem \ref{Birkhoff} naturally comes from the contraction mapping theorem, applying to arrays through the origin. It is also worth noting that in finding another, potentially simper, metric than the Hilbert metric such that a positive linear transformation is a contraction with respect to that metric, the authors in \cite{CMT} have shown that ``essentially, this is impossible.''

It has been customary to apply the Perron-Frobenius theorem, or its extension, the Krein-Rutman theorem in analysing the Poincare return map in a binary Glass network. We can see that the above generalisation is a stronger result than both of the theorems mentioned above, with one of its key advantage is that it does not require the dimension of the cone to be equal to the dimension of the linear space. This gives us some extra flexibility. Throughout this paper, the above result will be refered to as \textbf{the Birkhoff extension}.

Another lemma, which we often use to generate cones for applying the Birkhoff extension is
\begin{lemma}
\label{relint_convex_cone}
The relative interior of the convex cone in $\mathbb{R}^n$ generated by a non-empty convex set $C$ consists of the vectors of the form $\lambda x$ with $\lambda > 0$ and $x \in \relint C$.
\end{lemma}
A reference to this result can be found in the book ``Convex Analysis'' by R. Tyrrell Rockafellar \cite{convex_analysis}, just under corollary 6.8.1, page 50 .

Combining the results above, we have the following Lemma: 
\begin{lemma}
    \label{lemma_2.9}
    Let $M$ be a linear fractional map of the form $Mx = \frac{Ax}{1+\langle\psi,x\rangle}$ in $\mathbb{R}^n$, and let $\mathcal{M}$ be an affine subspace in $\mathbb{R}^n$ that does not contain the origin. If there is a convex and compact subset $\mathcal{H}$ of $\mathcal{M}$ such that $ M \mathcal{H} \subset \relint (\mathcal{H})$, then there is a unique and stable fixed point of $M$ inside $\mathcal{H}$.
\end{lemma}
\begin{proof}
Let $\mathcal{K}$ be the cone generated by $\mathcal{H}$, 
\[
\mathcal{K} = \bigcup_{\lambda \geq 0} \lambda \mathcal{H}.
\]
Extend the maps $A$ and $\mathbf{M}$ to all points in the cone $\mathcal{K}$ as follow: For a point $y \in \mathcal{K}\setminus\{0\}$, there is a unique $\lambda >0$ and a point $x \in \mathcal{H}$, such that $y = \lambda x$ (we have uniqueness because $\mathcal{H}$ lies in an affine subspace that does not contain the origin). Then define $A y = \lambda(Ax)$, and $My = \frac{\lambda(Ax)}{1+ \lambda \langle\psi,x\rangle}$. $A 0 = M 0 = 0$. Then easily see that by Lemma \ref{relint_convex_cone}, $A$ is a strictly positive linear mapping with respect to the cone $\mathcal{K}$. Thus, there must be a unique eigenvector of $A$ inside $\mathcal{K}$ with strictly dominant eigenvalue. We call this pair $(\lambda,v)$. Let $(v)$ be the halfline containing $v$ in $\mathcal{K}$, and let $w$ be the intersection of $(v)$ and $\mathcal{H}$. 
Now, $(v)$ is an invariant set of $M$, but so is $\mathcal{H}$. Hence, $w$ must be the unique fixed point of $M$ inside $\mathcal{H}$.
\end{proof}

\section{Overview of the techniques and results}

\subsection{Overview of the embedding technique}
Before trying to analyse the general Glass network, we will quickly review the key steps in analysing the original, binary Glass network.
\begin{enumerate}
    \item Firstly, we scale the decay rate to 1, which is just a linear transformation of the time parameter.
    \item Secondly, we make an affine transformation, so that the threshold of each variable is now $0$.
    \item Finally, we deduce a discrete map for transition between two walls in the form of a linear fractional map.
\end{enumerate}

\begin{figure}[!ht]
\centering
\includegraphics[width=0.9\textwidth]{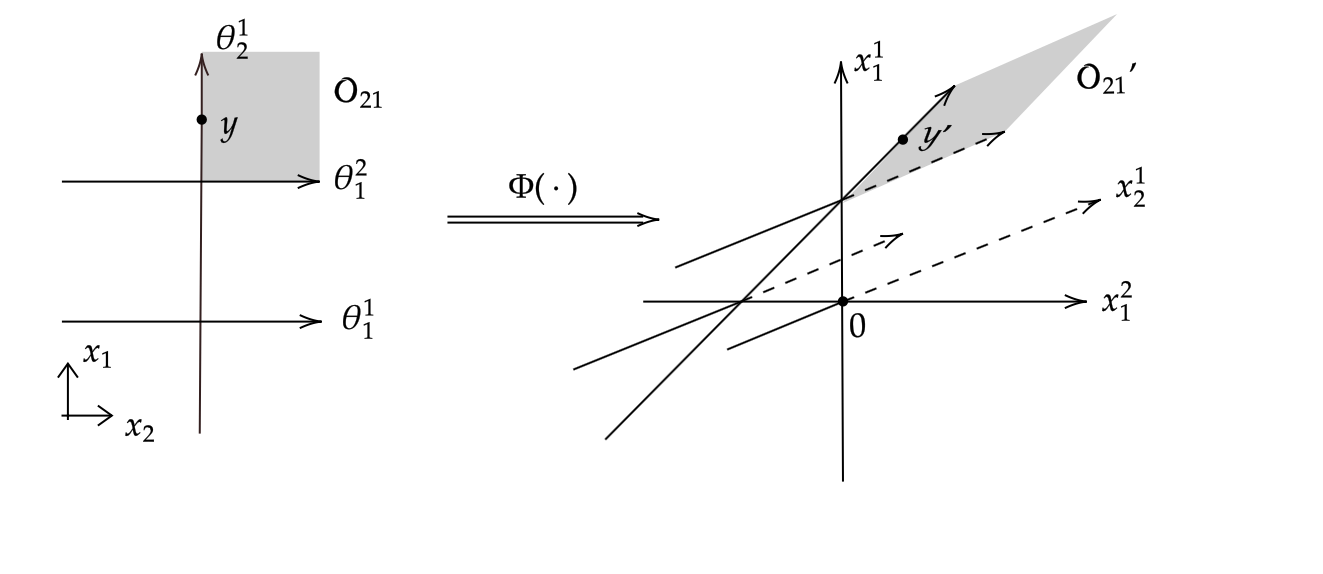}
\caption{\label{fig:diagram-2}A demonstration of our embedding function $\Phi$ that maps $y$ to $y'$, and the region $\mathcal{O}_{21}$ to $\mathcal{O}_{21}'$ (we are viewing $\mathcal{O}_{21}'$ upward). Notice that the region $\mathcal{O}_{21}'$ is fully contained in an orthant of the system composes of $(x_1^1,x_1^2,x_2^1).$}
\end{figure}

For a general network, the first step can be applied straight away, but the other steps cannot be copied trivially. For step 2, since a variable now can have two or more thresholds, there is no symmetrical choice of which threshold to scale to zero (while the others must be nonzero).

For step 3, the obvious thing to try would be to consider a switching between walls as locally a binary switching, then we can establish a discrete map between two immediate walls. Thus, a discrete map could theoretically be obtained by combining all the maps in between. However, there are a few theoretical and practical drawback to this approach:
\begin{enumerate}
    \item In a sequence of regions, we have to constantly perform affine transformations that map vertices to the origin, and then compose it with other maps. The result might be obtainable, but it is unclear whether that will be useful or not.
    \item In case there are more than one choice of vertices, there is no symmetrical way to choose one over the others.
    \item There is one particular switching that is novel to the general network: parallel switching, when trajectories starting at one wall enter a different wall that is parallel to the original one.
\end{enumerate}

The method about to be proposed here elegantly resolves all these problems, and in some rare instances, for theoretical purpose, it usage is even more advantageous compare to the original procedure for the binary networks. To be more specific, let us start with step 2.

Suppose we are given a variable $x_1$ with one threshold $\theta_1$. The standard procedure would be to take an affine transformation $x_1'=x_1-\theta_1$, and then the switching point $\theta_1'=\theta_1-\theta_1=0$. Now, if $x_1$ has two thresholds $\theta_1^1$ and $\theta_1^2$, instead of having to choose which $\theta_1^i$ to be transformed to zero, we can map both to zero, by creating two new variables, $x_1^1=x_1-\theta_1^1$, and $x_1^2=x_1-\theta_1^2$. Continue with this idea, we can define a new equivalent system. One can see that the new system (called the \textbf{embedded system}) should be equivalent to the old one, since coordinate-wise, this is just a normal affine transformation. Further, it is not hard to see that the thresholds of the variables $x_i^j$ are all zero now; thus, we find ourselves in a similar setting to a higher dimension binary Glass network.

However, the downside is we have to confine ourselves to an affine subspace of the higher dimension binary Glass network, since points outside of the subspace do not behave predictably, and are not very meaningful considering the original setting. Thus, results regarding the binary network, as well as techniques to prove them, have to be adapted to this affine subspace with care.

The best thing about this method is that it gives us the discrete map between walls as a linear fractional map, as derived in \cite{Edwards} and \cite{Glass_1978}. Thus, it provides a very nice form corresponds to step 3 listed above. Thus, similarly to the two papers above, periodic orbit analysis is reduced to determining all the fixed points of the return map, which corresponds to determining all the eigenvectors.

Finally, this method is in perfect harmony with the existing procedure for the binary Glass network: when restricted to the binary network, it is exactly the usual steps taken.

\subsection{Overview of the results regarding periodic orbits in a cyclic attractor}

The main results can be summarised as follow: given a cyclic attractor, either the periodic orbits are all degenerate, or there is a unique and asymptotically stable periodic orbit. We first refine the result for a non full-rank cyclic attractor in the binary network, then come to the two cases pertain to the general network: the ideal cyclic attractor, and non ideal cyclic attractor.

For a (non full-rank) cyclic attractor in the binary network, as an example, take the cyclic attractor $\mathcal{C}: \mathcal{O}_{000} \rightarrow \mathcal{O}_{100} \rightarrow \mathcal{O}_{110} \rightarrow \mathcal{O}_{010} \rightarrow \mathcal{O}_{000}$ as in figure \ref{fig:diagram-4}. 
Notice that the third coordinate is redundant, we can remove it to a full rank cyclic attractor $\mathcal{C}': \mathcal{O}_{00} \rightarrow \mathcal{O}_{10} \rightarrow \mathcal{O}_{11} \rightarrow \mathcal{O}_{01} \rightarrow \mathcal{O}_{00}$. Thus, we can apply Theorem \ref{theorem_2.2} to $\mathcal{C}'$. If $\mathcal{C}'$ is degenerate, then so must also be $\mathcal{C}$. Otherwise, a fixed point of the return map of $\mathcal{C}'$ corresponds to an invariant set of the return map of $\mathcal{C}$, and we can apply Lemma \ref{lemma_2.9} to find a unique periodic orbit inside $\mathcal{C}$. A rigorous treatment of this case can be found in subsection \ref{ssection_5.1}.

For the non ideal case, given a cyclic attractor $\mathcal{C}$, it behaves similar to a cyclic attractor $\mathcal{C'}$
in the binary case when restricted to a suitable domain. Our main difficulty is to show that the unique periodic orbit of the binary network, if it exists, has to lie in the smaller cyclic attractor of the general Glass network as well. But this must be true because the regions enclosed by $\mathcal{C}$ is invariant.

For the ideal case, apply Lemma \ref{lemma_2.9} to a suitable subset of the Poincare section, we can always find a unique and stable periodic orbit inside an ideal cyclic attractor.

\section{Rigorous construction of the embedding trick}
Let $m= \sum j_i$, we define an embedding function $\Phi$ from $\mathbb{R}^n$ to $\mathbb{R}^m$ as followed:
\[
\Phi (x) = \begin{bmatrix}    
    1 & 0 & \cdots & 0 \\
    1 & 0 & \cdots & 0 \\
    \vdots & \vdots & \ddots & \vdots \\
    1 & 0 & \cdots & 0 \\
    0 & 1 & \cdots & 0 \\
    \vdots & \vdots & \ddots & \vdots \\
    0 & 0 & \cdots & 1
    \end{bmatrix} x - 
    \begin{pmatrix}    
    \theta_{1}^1 \\ \theta_{1}^2 \\ \vdots \\ \theta_{1}^{J_1} \\ \theta_{2}^1 \\ \vdots \\ \theta_{n}^{J_n}
    \end{pmatrix}
\]
In the above formula, denote the matrix as ${B}$ and the vector as ${v}$, we can shorten the formula as
\[
\Phi (x) = \mathbf{B} x - \mathbf{v}.
\]
Let $x' = \Phi (x)$, then $(x')_i^j=x_i - \theta_{i}^j$, hence ${x'}_i^j$ keeps track of the relationship of $x_i$ and $\theta_i^j$. Call the phase space of the original system (of $n$ variable $x_1$, $...$, $x_n$) $\mathcal{M}$. Let $\mathcal{M}' = \Phi (\mathcal{M})$, clearly, from observing the form of $\Phi$, we have that $\mathcal{M}'$ is an affine subspace of dimension $n$ in $\mathbb{R}^m$, and that $\Phi$ is an affine mapping from $\mathbb{R}^n$ to $\mathbb{R}^m$, and a bijection from $\mathcal{M}$ to $\mathcal{M}'$. We endow $\mathbb{R}^m$ with a binary Glass network centred at the origin. Given an orthant $\mathcal{O}' \subset \mathbb{R}^m$, we define its focal point as follow:
\begin{enumerate}
    \item If $\mathcal{O}' \cap \mathcal{M}' = \emptyset$, pick its focal point $f'$ to be any point inside itself. (Actually, the dynamics of trajectories inside this orthant is irrelevant to us).
    \item If $\mathcal{O}' \cap \mathcal{M}' \neq \emptyset$, let $x'$, $y'$ be two arbitrary points in $\mathcal{O}' \cap \mathcal{M}'$. Define two points $x$, $y \in \mathbb{R}^n$ as follow $x:= \Phi^{-1} (x')$, and $y:= \Phi^{-1} (y')$. Checking coordinate-wise, it is obvious that these two points must be in the same region in $\mathbb{R}^n$. Further, if $z$ is a point in this region, then $\Phi(z)$ must be in $\mathcal{O}'$. Thus, we can find a unique region $\mathcal{O}$ in $\mathbb{R}^n$ such that $\Phi(\mathcal{O}) \subset \mathcal{O}'$. Let $f$ be the focal point of $\mathcal{O}$, define $f':=\Phi(f)$ to be the focal point of $\mathcal{O}'$.
\end{enumerate}

\subsection{Proof that the embedding trick is well-defined}

Our goal is to show that $\mathcal{M}'$ behaves ``similarly'' to $\mathcal{M}$. To do that, we need to show a few things
\begin{enumerate}
    \item $\Phi$ maps a region in $\mathbb{R}^n$ into a single orthant in $\mathbb{R}^m$, and a wall in $\mathbb{R}^n$ into a single wall in $\mathbb{R}^m$. This is required so that we can compare the discrete maps between two systems.
    \item $\Phi(x(t)) = \Phi(x_0)(t)$, where $x(t)$ is the trajectory in $\mathbb{R}^n$ starting at $x_0$, and $\Phi(x_0)(t)$ is the trajectory in $\mathbb{R}^m$ starting at $\Phi(x_0)$.
\end{enumerate}

(1) can be shown by checking all the coordinates. We will give here a proof for (2).

\begin{prop}
    $\Phi(x(t)) = \Phi(x_0)(t)$.
\end{prop}

\begin{proof}
    For any pair of $(i,j)$ satisfying $1\leq i \leq n$, $1\leq j \leq J_i$, we have
    \[    \Phi ({y}(t))_{i}^j = {y}_i (t) - \theta_{i}^j.\]
    And
    \[  {y'}_{i}^j(0)=\Phi ({y}(0))_{i}^j = {y}_i (0) - \theta_{i}^j.\]
    \[  {f'}_{i}^j=\Phi ({f})_{i}^j = {f}_i- \theta_{i}^j.\]
We have $ \dot {y'}_{i}^j (t) = {f'}_{i}^j - {y'}_{i}^j (t)$ for all $t \geq 0$.
Hence,
\begin{align}
{y'}_{i}^j (t) &= {f'}_{i}^j + ({y'}_{i}^j (0)- {f'}_{i}^j) e^{-t}\\
&= (f_{i}^j -\theta_{i}^j)  + (y_{i}^j (0)- f_{i}^j) e^{-t} \\
&= [f_{i}^j+ (y_{i}^j (0)- f_{i}^j) e^{-t}]-\theta_{i}^j \\
&= y_{i}^j(t) -\theta_{i}^j \\
&= \Phi( y_{i}^j(t)) .
\end{align}
Since this is true for any pair of $(i,j)$ satisfying the condition, we have that ${y'}(t) = \Phi ({y}(t))$, as long as ${y}(t)$ and ${y'}(t)$ stay in their respective region. But our calculation shows that they must be crossing boundaries and sub-boundaries at the same time, hence, we can extend the results to the next corresponding region and sub-region. Thus, we deduce that the formula holds for all $t \geq 0$.
\end{proof}

Observing that all the variables in the new system now have $0$ as their thresholds. This allows us to express the discrete maps between walls in $\mathbb{R}^m$, restricted to the space $\mathcal{M}'$ as linear fractional mapping in the form deduce in \cite{Edwards}
\[
{y_k} = \mathbf{M } y_0=\frac{A y_0}{1+\langle \psi,y \rangle}
\]
for the $k$th boundary crossing after ${y_0}$, where $A$ is an $m\times m$ matrix, and $\psi$ is an $m$ dimensional column vector. Further, if $x'$ is a point on the domain, and $t'$ is the time taken for the trajectory starting at $x'$ to reach the codomain, then 
\[
e^{t'} = 1+\langle \psi,x \rangle.
\]

\subsection{Finding fixed points in the Poincare return map}
\label{finding_fp}
A fixed point in the Poincare return map (assume that the point lies in the returning region) corresponds to a periodic orbit in the affine subspace $\mathcal{M}'$, which in turn corresponds to a periodic orbit in the original system $\mathcal{M}$.

We can proceed to analyse the stability of these fixed points, if they exists, and the periods of periodic orbits (as time is identical in both systems) in an almost identical fashion to subsections 4.3 and 4.4 in \cite{Edwards}.

As a remark, one could always create ``fake'' thresholds on any variable, as long as the focal points are invariant to this threshold, then the two systems behave identically. But a nice property of this trick is that if applies to the original Glass network, we can embed it in an affine subspace in a higher dimension binary Glass network, such that the affine subspace does not contain the origin (the origin cannot be contained in any embedded system that has at least two distinct thresholds on a single variable, since otherwise, we would deduce that the two thresholds are in fact equal, which is a contradiction). \textbf{Hence, without loss of generality, we always assume that the affine subspace which contains our embedded phase space does not contain the origin.} The implication of this assumption is that no two distinct points on the embedded phase space lies on the same line through the origin. This allows us to apply Lemma \ref{lemma_2.9} to a wide range of situations.

\section{Cyclic attractors}

\subsection{Cyclic attractors in the binary network}
\label{ssection_5.1}
Given a cyclic attractor $\mathcal{C}$ with the usual setting in a binary Glass network centered at $0$. Suppose further that it is non full-rank, meaning, there is at least one variable that does not switch. Without loss of generality, suppose that the $m$ variables that switch at least once in $C$ are $x_1$, $x_2$, $\dots$, $x_m$, and the variables $x_{m+1}$, $\dots$, $x_n$ are all non-negative in $C$.

First off, we notice that the variables $x_{m+1}$, $\dots$, $x_n$ are ``inconsequential''. To formalise what we mean by this idea, for points inside $\mathcal{C}$, we define an equivalent relation as follow: $y \equiv z$ if $y_i=z_i$ for $1\leq i \leq m$.

\begin{figure}[!ht]
\centering
\includegraphics[width=0.9\textwidth]{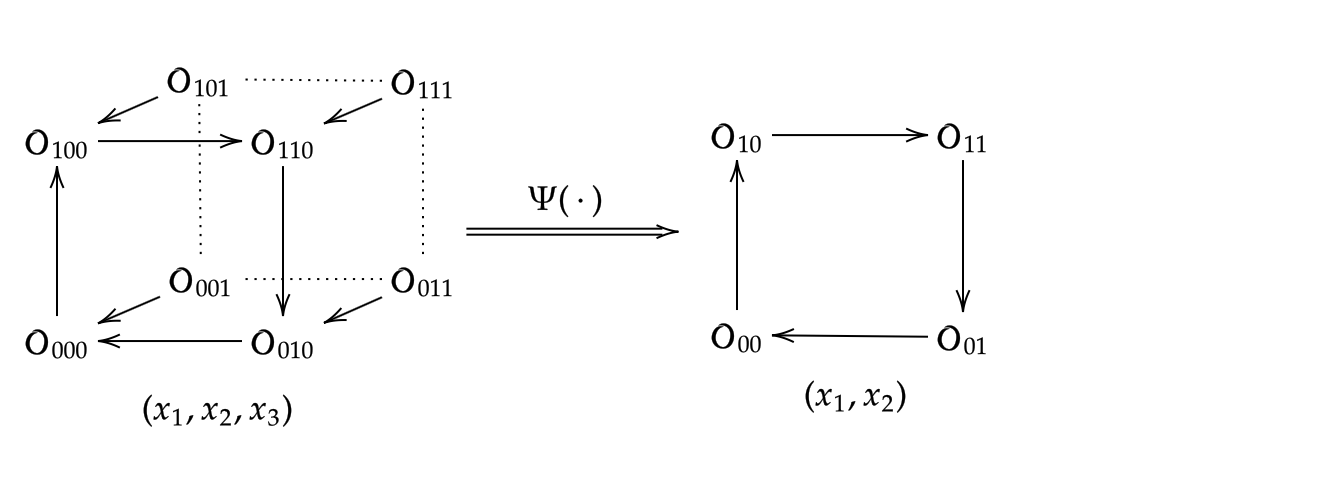}
\caption{\label{fig:diagram-4} The cyclic attractor $\mathcal{O}_{000} \rightarrow \mathcal{O}_{100} \rightarrow \mathcal{O}_{110} \rightarrow \mathcal{O}_{010} \rightarrow \mathcal{O}_{000}$ is of rank $2$, while our system has dimension $3$. However, if we ``compress'' it using our function $\Psi$ that essentially removes the third coordinate, then the resulting cyclic attractor, $\mathcal{O}_{00} \rightarrow \mathcal{O}_{10} \rightarrow \mathcal{O}_{11} \rightarrow \mathcal{O}_{01} \rightarrow \mathcal{O}_{00}$ is a full rank one.}
\end{figure}

\begin{prop}
\label{prop_5.1}
    If $y \equiv z$, then for all $t \geq 0$, $y (t) \equiv z(t)$ ($t$ ranges from $0$ to $\infty$, except for the case of finite time convergence then $t$ ranges from $0$ to the finite time convergence.
\end{prop}
\begin{proof}
    From the setting, by checking each coordinate, $y$ and $z$ have to be in the same orthant.
    Without loss of generality, assume that $y$ and $z$ are in $\mathcal{O}_1$, and that $x_h$ is the variable that switches between $\mathcal{O}_1$ and $\mathcal{O}_2$, $1\leq kh\leq m$. 
    Then, as long as the trajectories starting at $y$ and $z$ remains in $\mathcal{O}_1$, for any $i$ between $1$ and $m$, the ODE equations of $y_i$ and $z_i$ are identical, namely:
    \[
    \Dot{y_i} = f^1_i - y_i,
    \]
    and,
    \[
    \Dot{z_i} = f^1_i - z_i.
    \]
    Further, $y_i(0) = z_i(0)$ (by assumption of equivalent), hence, by uniqueness of solution to ODE, we must have $y_i(t) = z_i(t)$, as long as the trajectories still remain in $\mathcal{O}_1$.
    In particular, this apply to $y_h$ and $z_h$. We have that the $h$th coordinate of the two trajectories switches to $0$ at the same time, thus, implicitly using continuous extension of $\mathcal{O}_2$ onto the wall $\mathcal{W}_1$, the two trajectories are still equivalent upon entering into $\mathcal{O}_2$.

    Thus, by induction, we obtain the conclusion: $y(t) \equiv z(t)$ as long as the trajectories are well defined.
\end{proof}

Further extending this idea, we notice that the ${m+1}$th, $\dots$, $n$th coordinates of the focal points are all ``inconsequential'' as well. This pave the way for us to define a ``compressing'' function and a new ``compressed'' system, whereby we keep only the first $m$ coordinates, and obtain a new binary Glass network in lower dimension. The equivalent cyclic attractor $\mathcal{C}'$ in the new system will have full-rank; thus, we can apply the theorem of Glass, Theorem \ref{theorem_2.2}.

We define a ``compressing'' function $\Psi: \mathbb{R}^n \rightarrow \mathbb{R}^m$ as follows:
    \[
    \Psi (x) = \begin{bmatrix}    
    1_{(1)} & 0 & \cdots & 0_m & \cdots & 0 \\
    0 & 1_{(2)} & \cdots & 0_m & \cdots & 0 \\
    \vdots & \vdots & \ddots & \vdots & \cdots & \vdots \\
    0 & 0 & \cdots & 1_{(m)} & \cdots & 0
    \end{bmatrix} x 
    = (x_1,x_2,\dots,x_m)^T.
    \]
Easily checked that for an orthant $\mathcal{O}_j$ in $\mathcal{C}$, $\mathcal{O}_j' := \Psi (\mathcal{O}_j)$ is also an orthant in $\mathbb{R}^m$. Further, $ \mathcal{O}_i' \neq \mathcal{O}_j'$ for $i\neq j$ (since the two orthants $\mathcal{O}_i$ and $\mathcal{O}_j$ differs in the first $m$ coordinates). Set the focal point ${f^j} '$ of $\mathcal{O}_j '$ to be $\Psi (f^j)$, and by checking the position of the focal points, one can see that the cycle of orthants $\mathcal{O}_1' \rightarrow \mathcal{O}_2' \rightarrow \dots \rightarrow \mathcal{O}_k' \rightarrow \mathcal{O}_1' $ is a cyclic attractor in $\mathbb{R}^m$. Call this cycle $\mathcal{C}'$. Further, $\mathcal{C}'$ is a full-rank cyclic attractor in $\mathbb{R}^m$. But before examining more this new cyclic attractor, we need to establish a few link between the two systems.

\begin{prop}
    If $x\in \mathcal{O}_i$ and $x' \in \mathcal{O}_i '$, $x' = \Psi(x)$. Denoted $x(t)$ the trajectory starting at $x$ in $\mathbb{R}^n$, and $x'(t)$ the trajectory starting at $x'$ in $\mathbb{R}^m$, $x'(t) = \Psi(x(t))$, for any time $t$ in the maximal solution of both $x$ and $x'$.
\end{prop}
\begin{proof}
    We prove this by checking component-wise, in similar fashion to the proof of Proposition \ref{prop_5.1}.
\end{proof}

Let $\mathcal{W}_0$ the Poincare section, and $M$ the Poincare return map of the cycle of orthants $\mathcal{C}$ in $\mathbb{R}^n$, and let $\mathcal{W}_0 '$ the Poincare section, and $M'$ the Poincare return map of the cycle of orthants $\mathcal{C}'$ in $\mathbb{R}^m$, applying the above proposition, we have that $\Psi (Mx) = M'( \Psi (x))$.
If $x^*$ is a fixed point of $M$, then 
\begin{align}
    M'( \Psi (x^ * )) &= \Psi (M x^*) \\
    &= \Psi (x^*),
\end{align}
And thus, $\Psi (x^*)$ is also a fixed point of $M'$.
Thus, our analysis henceforth depends greatly on what happens on the cyclic attractor $\mathcal{C}'$. 

\textbf{Case 1: The cyclic attractor $\mathcal{C}'$ has a unique periodic orbit}, corresponds to a unique stable fixed point of $M'$ on $\mathcal{W}_0 '$. Called this fixed point $y^*$.  

Then easily check the set $E = \Psi ^{-1} (y^*) \cap \overline{\mathcal{ \mathcal{W}_0}}$ is an invariant set of $M$. 
Clearly, any fixed point of $M$ will have to lie in $E$. We will now show that $E$ has a unique stable fixed point. Let $u=\max \{ |f_i| \}$, $v=\min \{ |f_i| \}$, where $\{ f \}$ is the set of focal points of the orthants of $\mathcal{C}$, $m+1 \leq i \leq n$. we must have that $u\geq v > 0$. 

Now, consider $T = (y^*)' \times [v/2,u+1]^{n-m}$, which is compact and convex, and a proper subset of $E$. There can be no other fixed point of $M$ outside $T$, since $T$ is a globally attracting set.

Further, since the variables $x_{m+1},\dots,x_n$ do not switch, if there is a $y$ on the relative boundary of $T$, it must go into the relative interior of $T$ after one iteration (one can check this componentwise). Thus, we have that $M T \subset \relint T$. Further, the affine hull of $T$ does not containt the origin, thus, by Lemma \ref{lemma_2.9}, we have that $T$ has a unique stable fixed point, which corresponds to a unique stable periodic orbit of $\mathcal{C}$.

\textbf{Case 2: The cyclic attractor $\mathcal{C}'$ has a degenerate periodic orbit.} Then all the trajectories in $\mathcal{C}$ will have to converge to the spine of $\mathcal{C}$, since $\Psi^{-1} (0) = \mathcal{S}(\mathcal{C})$.

In summary,
\begin{theorem}
\label{theorem_5.3}
    Given a cyclic attractor $\mathcal{C}$ in $\mathbb{R}^n$, we have two possibilities:
    \begin{enumerate}
        \item The cyclic attractor has a unique and asymptotically stable periodic orbit.
        \item All trajectories in the cyclic attractor will converge to its spine (degenerate periodic orbits).
    \end{enumerate}
    The cases, as well as the periodic orbit, can be determined and calculated. 
\end{theorem}

\begin{remark}
    The conclusion here is a slight refinement of the conclusion in section 4.5, \cite{Edwards}. In a non full-rank cyclic attractor in the binary network, the matrix of the return map is not necessarily strictly positive (which is why in \cite{Glass_1978}, the authors only consider full-rank cyclic attractor). Hence, in the case the dominant eigenvalue(s) have absolute value greater than $1$, applying the Perron-Frobenius theorem straight away would only result in the existence of a neutrally stable periodic orbit. Our Theorem \ref{theorem_5.3} is slightly stronger, since we know in that case, there must be a unique and asymptotically stable periodic orbit.
\end{remark}

\subsection{Ideal cyclic attractor in a general Glass network}
\label{ssection_5.2}
\begin{prop}
\label{prop_5.5}
    Let $\mathcal{C} = \bigcup \mathcal{O}_i$ be any cyclic attractor. If there is a point $x \in \overline{ \mathcal{W}_i}$ such that $\mathbf{M} x \in \partial \mathcal{W}_{i+1}$, then $x$ must be a point on the spine of $\mathcal{C}$.
\end{prop}
\begin{proof}
    This is an application of Proposition \ref{prop_2.4}. We prove this by considering coordinate wise: in the coordinates that do not switch, the value of the trajectories that start on $\partial \mathcal{W}_i \setminus \mathcal{S}(\mathcal{C})$ must be strictly inside the range of the same coordinate of $\mathcal{W}_{i+1}$.
\end{proof}

\begin{corollary}
    In an ideal cyclic attractor, $\mathbf{M} \overline{\mathcal{W}} \subset \relint (\mathcal{W})$.
\end{corollary}
\begin{proof}
    If $x$ is a point in $\overline{\mathcal{W}}$ such that $\mathbf{M} x \in \partial{\mathcal{W}}$, then inductively, by Proposition \ref{prop_5.5}, $x \in \mathcal{S}(\mathcal{C})$. But this is a contradiction since in an ideal cyclic attractor, $\mathcal{S}(\mathcal{C}) = \emptyset$.
\end{proof}

\begin{figure}[!ht]
\centering
\includegraphics[width=0.9\textwidth]{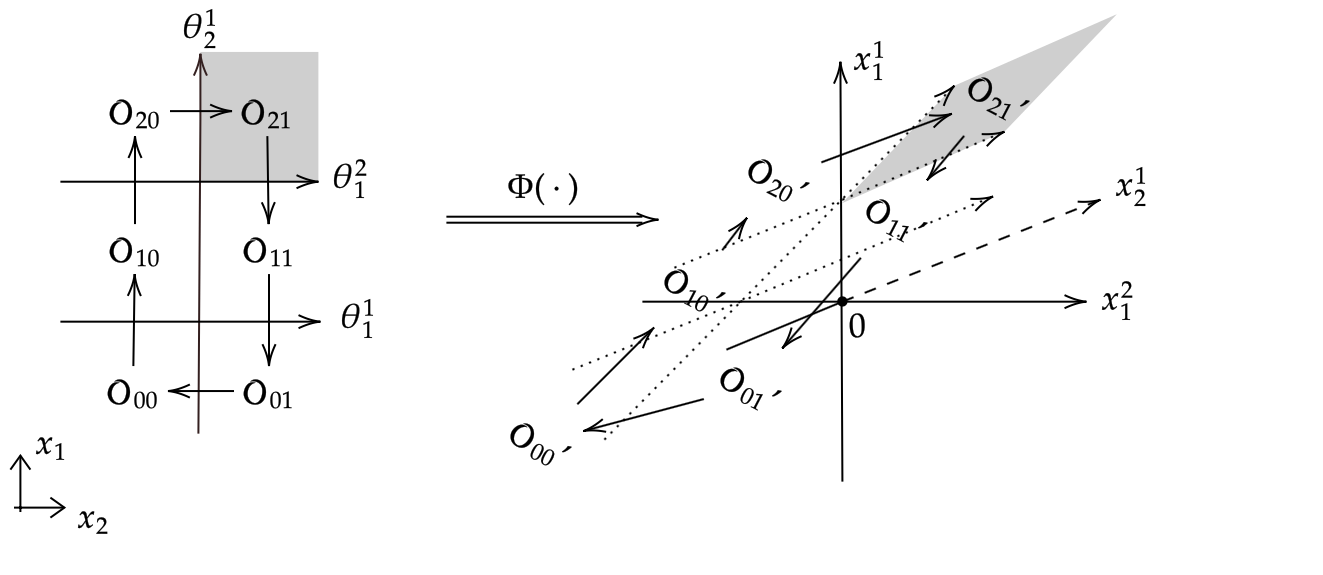}
\caption{\label{fig:diagram-6} The cyclic attractor $\mathcal{O}_{00} \rightarrow \mathcal{O}_{10} \rightarrow \mathcal{O}_{20} \rightarrow \mathcal{O}_{21} \rightarrow \mathcal{O}_{11} \rightarrow \mathcal{O}_{10}\rightarrow \mathcal{O}_{00}$ is an example of an ideal cyclic attractor.}
\end{figure}

\begin{figure}[!ht]
\centering
\includegraphics[width=0.9\textwidth]{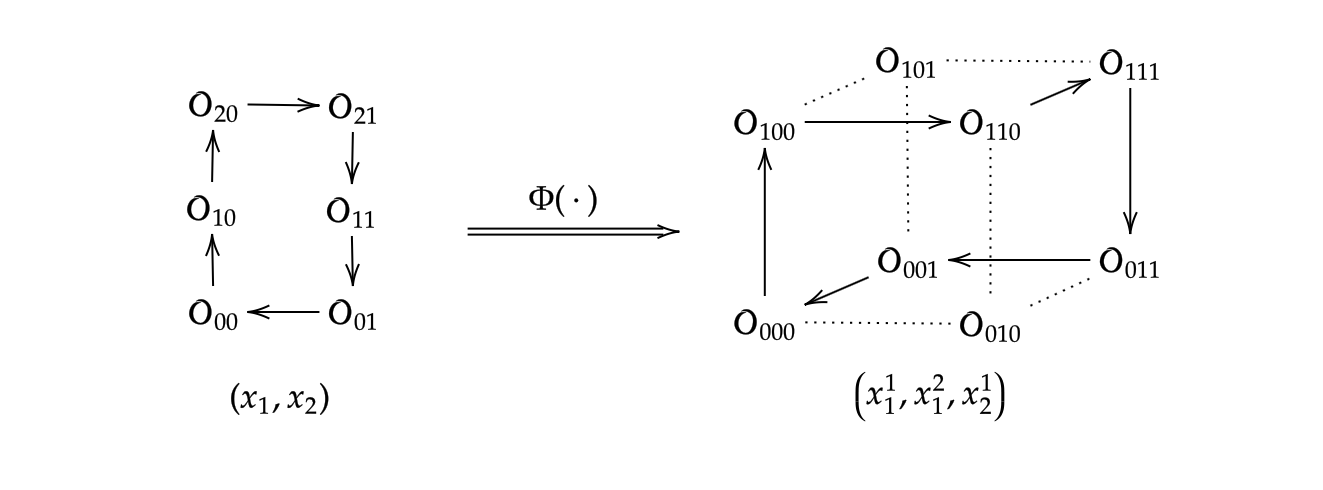}
\caption{\label{fig:diagram-3} The State Transition Diagraphs of the cyclic attractor in \ref{fig:diagram-6}, and of the embedded ``cyclic attractor'' on the affine subspace.}
\end{figure}

Then if the wall $\mathcal{W}$ is bounded, by Lemma \ref{lemma_2.9}, we must have a unique and stable fixed point inside $\mathcal{W}$. The case the wall $\mathcal{W}$ is unbounded requires a small and easy fix, which is relate to the property that the trajectories in a Glass network is asymptotically bounded. We will construct a box $\mathcal{H}$ in $\mathcal{W}$, such that every point in $\mathcal{W}$ will enter $\mathcal{H}$ if we repeatly apply the discrete map $\mathbf{M}$. The construction is as follow:

    Without loss of generality, assume $\mathcal{W}$ is a wall on variable $x_n$, for $x_n = \theta_{n}^j$. If the coordinate $x_1$ of $\mathcal{W}$ is bounded, pick $x_{1L}$ and $x_{1H}$ as the two end points of $x_1$. If $x_1$ is not bounded, without loss of generality, assuming in the positive direction, pick $x_{1L} = \theta_{1}^{J_1}$ as the low end point (the maximum switching point of $x_1$, pick $x_{1H} = max\{ f_1: \text{ $f$ is any focal point of $\mathcal{M}$}\} +1 $ as the high end point. (we want the upper bound to be strictly higher than the highest focal points in coordinate $x_1$).
    
    The hyperplane $ x_1 = x_{1H}$ divides $\mathcal{W}$ into 
    \[
    \mathcal{W}_1 = \{y \in \mathcal{W} : y_1 \leq x_{1H} \}
    \] 
    
    and $\mathcal{W}_2 = \mathcal{W} \setminus \mathcal{W}_1$.
    Given a point $\mathbf{y} \in \mathcal{W}$, if $y_1 \leq x_{1H}$, $(\mathbf{M y})_1 < x_{1H}$ (the trajectory must be moving away from this hyperplane, since all the focal points are strictly contained in the other hyperplane), hence $\mathcal{W}_1$ must map into itself. Further, easily check that no points in $\mathcal{W}_1$ would map into the relative boundary of either $\mathcal{W}$ or the hyperplane $ x_1 = x_{1H}$. Thus, we have that $\mathbf{M}$ maps $\mathcal{W}_1$ into its relative interior.

    The $x_1$ coordinate of any point in $\mathcal{W}_2$ will be strictly decreasing in the return map, hence, there can be no fixed point in $\mathcal{W}_2$. If If $x_1$ is not bounded in the negative direction, we treat it similarly. Repeat the process for other $x_i$, for the subset of $\mathcal{W}$ that maps into its relative interior, and we have constructed a suitable box $\mathcal{H}$ that have its vertices: $\prod_{i=1}^{n-1} \{ x_{iL} \text{ or } x_{iH} \}\times \theta_{n}^j$.

Thus, applying Lemma \ref{lemma_2.9} to the box $ \mathcal{H}$, we arrive at the following Theorem:
\begin{theorem}
\label{theorem_5.7}
    Given an ideal cyclic attractor $\mathcal{C}$ in $\mathbb{R}^n$, the cyclic attractor always has a unique and asymptotically stable periodic orbit.
\end{theorem}

\subsection{Non-ideal cyclic attractor in a general Glass network}
\label{ssection_5.3}
\begin{figure}[!ht]
\centering
\includegraphics[width=0.9\textwidth]{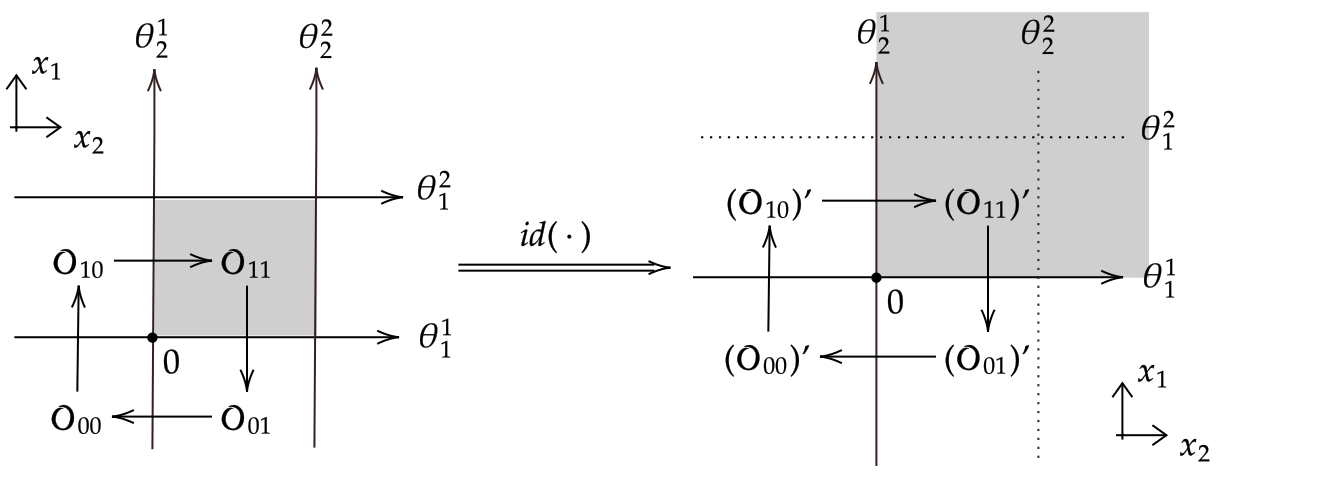}
\caption{\label{fig:diagram-5} The cyclic attractor $\mathcal{C}: \mathcal{O}_{00} \rightarrow \mathcal{O}_{10} \rightarrow \mathcal{O}_{11} \rightarrow \mathcal{O}_{01}\rightarrow \mathcal{O}_{00}$ is an example of a non ideal cyclic attractor. We ``embed'' it in a binary Glass network of the same dimension, such that the dynamics of the two are identical when restricted to the region $\mathcal{O}_{00} \cup \mathcal{O}_{10} \cup \mathcal{O}_{11} \cup \mathcal{O}_{01}$. Notice that while $\mathcal{C}':\mathcal{O}_{00}' \rightarrow \mathcal{O}_{10}' \rightarrow \mathcal{O}_{11}' \rightarrow \mathcal{O}_{01}'\rightarrow \mathcal{O}_{00}'$ is also a cyclic attractor, the two are not the same: for example, the region $\mathcal{O}_{11}'$ is much larger than $\mathcal{O}_{11}$. Our task is to show that if there is a periodic orbit in $\mathcal{C}'$, then that periodic orbit must also lie in $\mathcal{C}$. The reverse direction is trivial.}
\end{figure}

We now turn to the non-ideal case, which resembles a original cyclic attractor. Indeed, this would be the strategy we employ.
We use our usual setting of a cycle $C$ from $\mathcal{O}_1 \rightarrow \mathcal{O}_2 \rightarrow \dots \rightarrow \mathcal{O}_k \rightarrow \mathcal{O}_1 $.

Now, by Proposition \ref{prop_2.5}, there is at least one vertice $a$ of the network, such that $a$ belongs to all the closures of regions in $\mathcal{C}$. Without loss of generality, let $a=0$ (otherwise, we can just transform affinely our system by $-a$). Let ${(\mathbb{R}^n)}'$ be an $n$ dimensional real space distinct from $\mathbb{R}^n$. Let $id$ be a function from $\mathbb{R}^n$ to ${(\mathbb{R}^n)}'$, $id(x) = x$ (this is just a convenient function that help us determine which system we are working in). 

If $\mathbb{R}^n$ is our old system, we will define a new binary Glass network in ${(\mathbb{R}^n)}'$ (to distinguish it from the old system) as follow:
\begin{enumerate}
    \item If an orthant $\mathcal{O}'$ of ${(\mathbb{R}^n)}'$ fully contains a region $id(\mathcal{O}_i)$, denote that orthant $\mathcal{O}'_i$, and its focal point is the focal point of $\mathcal{O}_i$, ${f'}^i = id(f^i)$.
    \item Otherwise, that orthant is unimportant to our analysis, and we can pick its focal point to be any point satisfying the pre-requisite conditions.
\end{enumerate}

One could think about the two systems $\mathbb{R}^n$ and ${(\mathbb{R}^n)}'$ as overlapping each other, albeit distinct from each other. However, they behaves identically in the regions correspond to $\mathcal{C}$, $id(\mathcal{C})$ (which we will prove shortly). We then show that the cycle of orthants $\mathcal{C}'$ in ${(\mathbb{R}^n)}'$, $\mathcal{O}'_1 \rightarrow \mathcal{O}'_2 \rightarrow \dots \rightarrow \mathcal{O}'_k \rightarrow \mathcal{O}'_1 $ is a cyclic attractor in ${(\mathbb{R}^n)}'$. Thus, we can use results we have proven for this case in Section 5.1. We then attempt to reconcile the possible cases of Theorem 5.3 with the original cyclic attractor $\mathcal{C}$ in ${\mathbb{R}^n}$.

\begin{prop}
    If $x \in \mathcal{C}$, then $id(x(t))=(id(x))(t)$, where $(id(x))(t)$ is the trajectory starting at $id(x)$ in ${(\mathbb{R}^n)}'$, for any time $t$ in the maximal solution of both $x$ and $x'$
\end{prop}
\begin{proof}
    Without loss of generality, let $x\in \mathcal{O}_1$. Then $id(x) \in id(\mathcal{O}_1) \subset \mathcal{O}'_1$. Since the ODE equations of $x$ and $id(x)$ are identical, by uniqueness of solution to ODE, we have that $id(x(t))=(id(x))(t)$, as long as $x(t)$ remains in $\mathcal{O}_1$, and $x'(t)$ remains in $\mathcal{O}'_1$. But easily see that the two trajectories switch to $\mathcal{O}_2$ and $\mathcal{O}'_2$ at the same time, denote $t^*$, and at that time, $id(x(t^*))=(id(x))(t^*)$. Thus, by induction, we have that $id(x(t))=(id(x))(t)$, as long as two trajectories are well-defined.
\end{proof}

\begin{lemma}
    The cycle of orthants $\mathcal{C}'$ is a cyclic attractor in ${(\mathbb{R}^n)}'$.
\end{lemma}
\begin{proof}
    One can see this easily by considering the coordinates of each ${f'}^i$. Specifically, one can show that ${f'}^i \in \mathcal{O}'_{i+1}$, and thus, $\mathcal{C}'$ is a cyclic attractor.
\end{proof}

Let $\mathcal{W}_0$ the Poincare section, and $M$ the Poincare return map of the cycle of orthants $\mathcal{C}$ in $\mathbb{R}^n$, and let $\mathcal{W}_0 '$ the Poincare section, and $M'$ the Poincare return map of the cycle of orthants $\mathcal{C}'$ in ${(\mathbb{R}^n)}'$. 

\begin{prop}
     If $x^*$ is a fixed point of $M$, then $id(x^*)$ is a fixed point of $M'$.
\end{prop}
\begin{proof}
    The trajectory starting at $x^*$, $x^*(t)$, is contained in $\mathcal{C}$, thus, by Proposition 5.4, $id(x^*(t))=(id(x^*))(t)$. Notice that $id(\mathcal{W}_0) \subset \mathcal{W}'_0$; also, from the proof of Proposition 5.4, $x^*(t)$ and $id(x^*)(t)$ return to $\mathcal{W}_0$ and $\mathcal{W}_0 '$ after one cycle at the same time, denote this time $t^*$. Thus, we have
    \begin{align}
        M' (id(x^*)) &= id(x^*)(t^*) \\
        &= id(x^*(t^*)) \\
        &= id(M x^*) \\
        &= id(x^*),
    \end{align}
    $id(x^*)$ must be a fixed point of $M'$.
\end{proof}

Thus, once we have identified all the possible fixed point of $M'$, we can use the function $id^{-1}$ to examine which point could be a fixed point of $M$. But $M'$ is a Poincare return map of the cyclic attractor $\mathcal{C}'$, and by Theorem \ref{theorem_5.3}, either $M'$ has a unique and stable fixed point in $\mathcal{W}_0'$, or it has no fixed point. The latter immediately implies that $M$ has no fixed point in $\mathcal{W}_0'$, but what about the potential fixed point in $\mathcal{W}_0'$? We have that $id(\mathcal{W}_0) \subset \mathcal{W}'_0$, and this point (after applying $id^{-1}$) might lie outside of $\mathcal{W}_0)$. We shall see that this cannot be the case.

The map $M'$ can be expressed in the form of a linear fractional map as $M' x = \frac{A x}{1+ \langle \psi,x,\rangle}$, with $A$ an $n\times n$ matrix, and $\psi$ a column vector of length $n$, since the binary Glass network in ${(\mathbb{R}^n)}'$ is centred at $0$. If there is a fixed point $x^*$ inside $\mathcal{W}_0'$, by proposition 14 in \cite{Edwards}, it must lie on an eigenvector $(v)$ of $A$ that corresponds to an eigenvalue $\lambda$.

Assume that $x^*$ lies outside $id(\overline {\mathcal{W}_0 })$, let $x'$ be the intersection different from $0$ of $(v)$ and $id(\overline {\mathcal{W}_0 })$ (this point must exist because $(v)$ lies in the interior of $\mathcal{W}_0'$). \\
Since $\mathcal{C}$ is a cyclic attractor, and $id^{-1}(x') \in \overline{\mathcal{C}}$, implicitly continuous extending the map $M$ to $\overline {\mathcal{W}_0 }$, we have that $M (id^{-1}(x')) \in \overline{\mathcal{W}_0}$. Thus, $M' x' \in id(\overline {\mathcal{W}_0 })$. But from the analysis of the dynamic on the eigenvector $(v)$ carried out in \cite{Edwards} (proposition 15), $M' x'$ has to move toward $x^*$, thus, it must lie strictly outside $\overline{\mathcal{W}_0}$, a contradiction. Hence, this cannot be the case.

$x^*$ cannot lie on the relative boundary of $id( {\mathcal{W}_0 })$, otherwise, translating back to the general Glass network in $\mathbb{R}^n$, this would mean there are two variables on the trajectory starting at $ id^{-1}(x^*)$ that switch at the same time. But this is a contradiction since $\mathcal{C}$ is a cyclic attractor. Hence, if there is a fixed point of $M'$, it must correspond to a fixed point of $M$. Consequently, for this case, we derived at the same conclusion as in Theorem \ref{theorem_5.3}.

\begin{theorem}
\label{theorem_5.11}
    Given a cyclic attractor $\mathcal{C}$ in a general Glass network, we have two possibilities:
    \begin{enumerate}
        \item The cyclic attractor has a unique and asymptotically stable periodic orbit.
        \item All trajectories in the cyclic attractor will converge to its spine (degenerate periodic orbits).
    \end{enumerate}
    The cases, as well as the periodic orbit, can be determined and calculated. 
\end{theorem}

\subsection{Summarising the results}
\label{ssection_5.4}
We will summarise the results of Theorems \ref{theorem_5.3}, \ref{theorem_5.7}, and \ref{theorem_5.11} into one big theorem that gives a complete classification of the topological features of the flow in any cyclic attractor in both the binary and the general Glass network.

\begin{theorem}
\label{theorem_5.12}
    Given a cyclic attractor $\mathcal{C}$ in a general Glass network, the topological features of the flow depends on the classification of $\mathcal{C}$. If $\mathcal{C}$ is an ideal cyclic attractor, $\mathcal{C}$ always has a unique and asymptotically stable periodic orbit. If $\mathcal{C}$ is a non-ideal cyclic attractor, there are two possibilities:
    \begin{enumerate}
        \item The cyclic attractor has a unique and asymptotically stable periodic orbit.
        \item All trajectories in the cyclic attractor will converge to its spine (degenerate periodic orbits).
    \end{enumerate}
    The cases, as well as the periodic orbit, can be determined and calculated. 
\end{theorem}

\begin{remark}
    The result of this work (Theorem \ref{theorem_5.12}) slightly overlap with a work of Etient, \cite{Etienne_2009}. In that paper, the authors allow many thresholds on each variable, as well as non-uniform decay rate for each variable. Thus, the class of models they examined is larger than ours. However, in order to arrive at a similar conclusion, the authors impose restriction on focal points: they now have to satisfy a condition called alignment. I would also like to acknowledge here an earlier work of Snoussi, \cite{Snoussi}, which is among the first works that examine the non-uniform decay rate.
\end{remark}

\section{Practical application aspect of our method}

Given a sequence of regions (which includes a cycle in the regions), for practical, computational purpose, it will be more computationally optimised to only ``decompose'' the thresholds that are involved, and leave the rest as they are. For example, in a non ideal cyclic attractor such as in Figure \ref{fig:diagram-5}, we can calculate explicitly the return map in $2$ dimension, without having to embed it in a $4$ dimension system.

\section{Acknowledgement}
This work is part of a summer research project that I carried out when I was at the University of Bath. I would like to express here my gratitude to Doctor Karol Bacik and Professor Jonathan Dawes for their help and guidance throughout the project, and the completion of this paper.

\bibliographystyle{plain}
\bibliography{main}
\end{document}